\documentclass[11pt]{article}

\usepackage{geometry}
\usepackage{amsmath}
\usepackage{amsthm}
\usepackage{amssymb}
\usepackage{tikz}
\usepackage[implicit]{hyperref}
\usepackage[bottom]{footmisc}
\usepackage[affil-it]{authblk}
\usepackage[toc,page]{appendix}

\newtheorem{thm}{Theorem}[section]
\newtheorem{prop}[thm]{Proposition}
\newtheorem{lem}[thm]{Lemma}
\newtheorem{cor}[thm]{Corollary}

\newtheorem*{thm*}{Theorem}
\newtheorem*{cor*}{Corollary}
\newtheorem*{claim*}{Claim}

\newtheorem{mainthm}{Main Theorem}

\theoremstyle{definition}

\newtheorem{rem}[thm]{Remark}

\newtheorem*{rem*}{Remark}

\newcommand{\Gal}{\mathrm{Gal}}
\newcommand{\Aut}{\mathrm{Aut}}

\newcommand{\n}{\circ n}
\newcommand{\m}{\circ m}

\newcommand{\xdownarrow}[1]{%
  {\left\downarrow\vbox to #1{}\right.\kern-\nulldelimiterspace}
}

\title{Maximal Arboreal Galois Images for Polynomials of Twisted Carlitz Type}

\author{MONA AL BATROUNI AND CHIEN-HUA CHEN \thanks{Electronic address: \texttt{mab108@mail.aub.edu,  cc45@aub.edu.lb}}}
 \affil{Department of Mathematics, American University of Beirut}

\date{}   % removes the date
\begin{document}
\maketitle
\abstract
In this paper, we study the arboreal Galois representations for polynomials of twisted Carlitz
type, whose first iterated Galois group is linked to the torsion of a twisted Carlitz module. We prove two explicit families of polynomials having iterated Galois groups isomorphic to full iterated cyclic wreath product at every level. We then compare the arboreal Galois image of a polynomial of twisted Carlitz type with the adelic Galois image of its corresponding twisted Carlitz module, and show that arboreal maximality and adelic surjectivity are logically independent, except for a one-way local implication.

\section{Introduction}
Arithmetic dynamics studies the interplay between iteration of algebraic maps
and arithmetic structures such as Galois groups, rational points, and reduction
modulo primes. A central object is the \emph{arboreal Galois representation}
attached to a polynomial or rational map, introduced by Odoni and developed into
an active theory through the work of Boston--Jones, Stoll, Pink, and many others;
Jones' survey \cite{J13} gives a broad account, emphasizing the role of critical orbits and primitive prime divisors in controlling the size of the image.

We recall the basic construction. Let $K$ be a field, let $f \in K[x]$ have
degree $d \geq 2$, and fix a base point, which throughout this paper is $0$. For
each $n \geq 1$ the roots of the iterate $f^{\circ n}(x)$ in $K^{\mathrm{sep}}$
form the $n$-th level of a rooted preimage tree, where each root $\beta$ of
$f^{\circ(n-1)}$ is joined to the solutions of $f(x) = \beta$. The absolute
Galois group $\Gal(K^{\mathrm{sep}}/K)$ permutes these roots compatibly with the
tree structure, giving a tower of finite quotients
\[
  \mathcal{G}_n = \Gal\!\big(\textrm{Spl}_K(f^{\circ n})/K\big)
\]
together with embeddings $G_n \hookrightarrow \Aut(T_n)$, and, in the inverse
limit, the \emph{arboreal Galois representation}
$\Gal(K^{\mathrm{sep}}/K) \to \Aut(T_\infty)$. We assume here that the
iterates $f^{\circ n}$ are separable, so that each level has $d^n$ vertices,
$T_n$ is a regular $d$-ary tree, and $\Aut(T_n)$ is the $n$-fold iterated wreath
product of the symmetric group $S_d$. A guiding problem is then to
determine when the image $G_n$ is as large as possible, or at least of finite
index; this is a dynamical analogue of measuring the image of a Galois
representation attached to a classical arithmetic object.

The bulk of this theory has been developed over number fields, or more generally in characteristic zero or characteristic different from two. Even there, full maximality at every level is the exception rather than the rule: results typically establish {finite index}. Bush, Hindes, and Looper \cite{BHL17},
for instance, prove that the arboreal representations of certain prime-degree
unicritical polynomials over $\mathbb{Q}(\zeta_p)$ have finite index in the
infinite iterated wreath product of cyclic groups, with full surjectivity
obtained only in small degree examples. By contrast, the positive characteristic
theory is far less developed, and the results that exist stop short of
maximality. Hindes and Jones \cite{HJ20}, studying Riccati equations and
polynomial dynamics over function fields, establish primitive divisor phenomena
and {finite index} arboreal results, while emphasizing that characteristic
$p$ introduces genuine new obstructions---isotriviality and inseparability chief
among them. From a different direction, Pink \cite{P13-1, P13-2, P13-3}
determined profinite iterated \emph{monodromy} groups for quadratic polynomials
and morphisms in arbitrary characteristic. These works show that
positive characteristic arboreal dynamics is both tractable and subtle, but none
establishes that the arithmetic arboreal image is the {full} iterated
wreath product at every level for an explicit family in characteristic $p$.

The purpose of this paper is to prove such maximality results, for a family
naturally connected to Drinfeld modules. Let $q=p^e$ be an odd prime power, $\mathbb F_q[t]$ be the polynomial ring, with field of fraction $K = \mathbb{F}_q(t)$.  Set $d = q-1$, and consider the polynomials of \textbf{twisted
Carlitz type}
\[
  f_c(x) = c\,x^{d} + t \in K[x].
\]
The identity $\rho_t^{c}(x) = x\, f_c(x)$ relates $f_c$ to the rank-one twisted
Carlitz module $$\rho_t^{c}(x) = tx + c\,x^{q}.$$ Thus the first level of the
arboreal tower is determined by the $t$-torsion points of $\rho_t^{c}$. Our first main results show, for suitable choice of $c$,
that arboreal maximality persists at {every} level.

\begin{mainthm}[Theorem \ref{thm:conditional-wreath}+Remark \ref{maxcond}]
    Let $q=p^e\geq 5$ be an odd prime power, $d=q-1$, and $K=\mathbb F_q(t)$. Assume that \(M+d-1\not\equiv 0 \pmod p\) and that the Kummer class \([-t^{1-M}]\in K^\times/(K^\times)^d\) has order \(d\). Then for every \(n\ge 1\), the n-th iterated Galois group of $f_M(x)=t^Mx^d+t$ satisfies \( \mathcal{G}_n \cong C_d \wr C_d \wr \cdots \wr C_d,\) with \(n\) many factors. In particular, one can attain arboreal maximality for $f_M(x)$ if and only if $M$ satisfies the additional condition below:
$$
\operatorname{lcm}\left(
2,
\frac{d}{\gcd(d,1-M)}
\right)
=d.
$$
\end{mainthm}

\begin{mainthm}[Theorem \ref{thm:conjugate-full-wreath}]
    Let \(q=p^e\ge 7\) be an odd prime power,  $K=\mathbb F_q(t)$, \(d=q-1\), and let
\[
f_c(x)=cx^d+t\in K[x].
\]
Assume that
\(c\in (K^\times)^{d-1},
\textrm{ and }
v_t(c)=0.\)
Choose \(u\in K^\times\) such that \(u^{d-1}=c^{-1}\), and set
\[
F(x)=u^{-1}f_c(ux)=x^d+C,
\quad \textrm{ with }
C=\frac{t}{u}.
\]
Denote 
\(
K_n=\operatorname{Spl}_K(F^{\circ n})=\operatorname{Spl}_K(f_c^{\n})
\quad\text{and}\quad
\mathcal{G}_n=\operatorname{Gal}(K_n/K).
\)
Then for every \(n\ge1\), we have
\[
\mathcal{G}_n\cong
\underbrace{C_d\wr C_d\wr\cdots\wr C_d}_{n\text{ factors}}.
\]
\end{mainthm}
Note that when $q\geq 7$, the two families intersect only at $M=0$ and $c=1$. In contrast to the finite-index results available in positive characteristic, our theorems yield the full iterated wreath product at every level for explicit families over $\mathbb F_q(t)$. The method combines three ingredients. A Kummer independence criterion
(Proposition~\ref{prop:general-kummer-independence}) reduces maximality at each level to the existence of a place isolating a single Kummer class; primitive simple divisor theorems for the critical orbit (Theorem~\ref{thm:PSD} and Proposition~\ref{lem:primitive_place_F}), proved via
the Mason--Stothers theorem over $\mathbb{F}_q[t]$; and
Capelli's lemma controls the irreducibility of iterates.

After this work was completed, we learned of the paper of Adams and 
Hyde~\cite{AH25}, which studies the profinite iterated monodromy group of a 
unicritical polynomial $f(x) = ax^d + b$ over a field $K$ with $d$ coprime to 
$\operatorname{char} K$ via the preimage tower of a \emph{transcendental} base 
point $s$, that is, the Galois theory of $f^{\circ n}(x) = s$ over $K(s)$. In 
the post-critically infinite case they show the geometric image is the full 
iterated cyclic wreath product $[C_d]^\infty$ (\cite[Prop.~3.14]{AH25}) and 
determine the constant field extension $\widehat{K}_{f,\ell} = K(\zeta_d)$ 
(\cite[Prop.~6.6]{AH25}). Regarding $f_c$ as a polynomial over 
$K = \mathbb{F}_q(t)$ and applying these results, the generic image is 
maximal at every level, the constant field extension being trivial since 
$\mu_{q-1} = \mathbb{F}_q^\times \subset \mathbb{F}_q$. Our tower is the 
specialization of theirs at $s = 0$: we study $f_c^{\circ n}(x) = 0$ over 
$\mathbb{F}_q(t)$, the base point $0$ being the $t$-torsion point of the 
twisted Carlitz module $\rho_t^c$. Maximality of the generic image does not 
imply maximality after specialization, as the arboreal image may shrink at a 
special point; the primitive simple divisor theorems in Sections~\ref{psd1} 
and~\ref{psd2}  are precisely what guarantee no such collapse occurs at $0$. Note that the specialization problem is not addressed in~\cite{AH25} (see their \S1.3).

Our second main result is the relationship between the arboreal Galois image of $f_c$ and the
adelic Galois image of the associated twisted Carlitz module $\rho_t^{c}$. Although
$\rho_t^{c}(x) = x f_c(x)$ links the two theories at the first level, their higher
towers encode different arithmetic. Using Gekeler's defect formula \cite{GEKELER2016316},
see Section 2.5.2, we provide examples showing that adelic surjectivity and arboreal
maximality are logically independent: we construct a twisted Carlitz module with
surjective adelic Galois image while its corresponding polynomial is arboreal non-maximal
(Proposition~\ref{adelsurj}+ Proposition\ref{nonmax}). Conversely, we construct a polynomial with maximal arboreal Galois image while its corresponding twisted Carlitz module has non-surjective adelic Galois
image (Proposition~\ref{arbmaxnonsurj}). The two given examples motivate us to explore the following local implication. 

\begin{mainthm}[Theorem \ref{arboreal-max}]
    Let $q=p^e$ be an odd prime power and $d=q-1$.
    Let \(\rho_t^c(x)=tx+cx^q\) be the twisted Carlitz module corresponding to $f_c(x)=cx^{q-1}+t$. Suppose that the first arboreal level
of \(f_c\) is maximal, i.e.
\[
\operatorname{Gal}(K(f_c^{-1}(0))/K)\cong C_d.
\]
Then the \(t\)-adic Galois representation attached to \(\rho_t^c\) is surjective, i.e, \(\operatorname{def}(\Delta,t^m)=1 \ \text{for every}\ m\ge1,\) where \(\Delta\in \mathbb F_q[t] \) is the \(d\)-power-free representative of the class of \(c\) in
\(K^\times/(K^\times)^d\). In other words, arboreal maximality of \(f_c\) implies \(t\)-adic surjectivity of the associated twisted Carlitz module $\rho_t^c$.
\end{mainthm}
Finally, in the appendix we record an eventual primitive simple divisor theorem for polynomial $f_c(x)=cx^d+t$ with $d\geq 3,\quad p\nmid d, \textrm{ and }c\in \mathbb F_q[t]\setminus \{0\}$. This result can serve as a function field arithmetic application of \cite{T13}

\section*{Acknowledgements}
C.-H.\ Chen was supported by the CAMS Fellowship for 2026--2027 cycle.

\section{Preliminaries}

\subsection{The Mason--Stothers Theorem over $\mathbb{F}_q[t]$}

The following is the polynomial form of the well-known function-field \(abc\) theorem, see~\cite{S00} for an elementary proof.

\begin{thm}[Mason--Stothers theorem]\label{thm:mason-stothers}
Let \(k\) be a field, and let \(A,B,C\in k[t]\) be nonzero pairwise coprime polynomials such that \(A+B+C=0.\)
Then either
\[
A'=B'=C'=0,\quad \textrm{or}\quad \max\{\deg A,\deg B,\deg C\}
\le
\deg \operatorname{rad}(ABC)-1.
\]
\end{thm}

Here \(\operatorname{rad}(F)\) denotes the product of the distinct irreducible factors of \(F\). In characteristic \(p>0\), the exceptional case \(A'=B'=C'=0\) is equivalent to \(A,B,C\in k[t^p]\).

\subsection{Arboreal Notation}

Throughout this paper, let \(q=p^e\) be an odd prime power, we further require $q\geq 5$ in section 4, and $q\geq 7$ in section 5. We write \(K=\mathbb{F}_q(t), \ \textrm{ and } \mathcal{R}=\mathbb{F}_q[t].\) Except in Appendix A, we set $d=q-1$, so \(\mu_d=\mathbb{F}_q^\times\). For each finite place \(P\) of \(K\), we denote by \(v_P\) the corresponding normalized discrete valuation. In particular, \(v_t\) denotes the valuation associated with the prime polynomial \(t\).\\
For a polynomial \(f(x)\in K[x]\), we denote by \(f^{\circ n}\) its \(n\)-fold iterate and by \(\operatorname{Spl}_K(f)\) its splitting field over \(K\). For each integer \(n\ge1\), we write \(\ R_n=\{\alpha\in\overline{K}:f^{\n}(\alpha)=0\}\) for the set of roots of \(f^{\circ n}(x)\). If \(K_n=\operatorname{Spl}_K(f^{\circ n})\), we write \(\mathcal{G}_n=\operatorname{Gal}(K_n/K)\). If \(L\) is a field and \(m\ge1\), we write \((L^\times)^m\) for the subgroup of \(m^{th}\) powers in \(L^\times\).\\
Finally, we recall the convention for semi-direct products and wreath products. Let a group \(H\) act on a finite set \(S\), and let \(G^S\) denote the direct product of copies of \(G\) indexed by \(S\). We let \(H\) act on \(G^S\) by permuting the factors: \[ (h\cdot a)_s=a_{h^{-1}s}, \quad \textrm{ where }h\in H,\ a=(a_s)_{s\in S}\in G^S. \] The corresponding semidirect product is denoted \( G^S\rtimes H, \) with multiplication \[ (a,h)(b,k)=\bigl(a(h\cdot b),hk\bigr). \]
In this paper, the relevant case is \(G=C_d\), where \(C_d\) is the cyclic group of order \(d\). If \(H\) acts on a finite set \(S\), we write \[ C_d\wr H:=C_d^S\rtimes H \] for the corresponding permutational wreath product. Iteratively, the full cyclic wreath product of height \(n\) is defined by \[ W_1=C_d,\textrm{ and } W_n=C_d^{S_{n-1}}\rtimes W_{n-1}\quad \textrm{ for } n\ge2, \] where \(S_{n-1}\) is the set of vertices at level \(n-1\) of the rooted \(d\)-ary tree, and \(W_{n-1}\) acts on \(S_{n-1}\) by its natural action. Thus \[ W_n\cong \underbrace{C_d\wr C_d\wr\cdots\wr C_d}_{n\text{ factors}}. \] When we say that the arboreal Galois group at level \(n\) is the full cyclic wreath product, we mean that \( \mathcal{G}_n\cong W_n. \)

\subsection{Kummer Theory}

We recall the form of Kummer theory that will be used throughout the paper; see
\cite[Chapter~8, Sections 8--9]{Lang2002}.

\begin{thm}[Kummer Theory]
\label{Kummer-Theory}
Let $L$ be a field and let $d\ge 2$ be an integer such that
$\operatorname{char}(L)\nmid d$ and $\mu_d\subset L$.
Let $a_1,\dots,a_r\in L^\times$ and set \(
E=L\!\left(\sqrt[d]{a_1},\dots,\sqrt[d]{a_r}\right).\)
If \(\Gamma=\langle [a_1],\dots,[a_r]\rangle\subset
L^\times/(L^\times)^d,\) then \(\operatorname{Gal}(E/L) \cong
\operatorname{Hom}(\Gamma,\mu_d).\) \\
In particular, if the classes \( [a_1],\dots,[a_r]\) are linearly independent in \(L^\times/(L^\times)^d,\) then \\
\(\operatorname{Gal}(E/L)\cong (C_d)^r.\)
\end{thm}

\begin{thm}
\label{kkummer-thm-used}
Let \(f_c(x) = cx^{d} + t\in K[x],\ c\in K^{\times}.\) Assume that \(f_c^{\circ m}(x)\) is separable over \(K\) for every \(1\leq m\leq n\). Then \(\mathcal{G}_{n}\) embeds naturally into the \(n\)-fold iterated wreath product \(C_{d}\wr C_{d}\wr \dots \wr C_{d}\) with \(n\) factors. More precisely, for every \(n\geq 2\), there is a natural embedding \[\mathcal{G}_{n}\hookrightarrow (C_{d})^{R_{n - 1}}\rtimes \mathcal{G}_{n - 1},\]
where \(\mathcal{G}_{n - 1}\) acts on \(R_{n - 1}\) by its natural action on the roots of \(f^{\circ (n-1)}(x)\).
\end{thm}

\begin{proof}
Since \(\mu_{d} \subset K\) and \(p\nmid d\), all \(d^{th}\) root extensions appearing below are tame Kummer extensions. At the first level, the equation \(f(x)=0\) is equivalent to \(x^{d} = -\frac{t}{c}\). Thus \(K_{1}/K\) is generated by a \(d^{th}\) root of an element of \(K^{\times}\). Since \(\mu_{d}\subset K\), Theorem \ref{Kummer-Theory} gives \(\mathcal{G}_{1}\hookrightarrow C_{d}\). Now let \(n\geq 2\). For each \(\beta \in R_{n-1}\), the points in the fiber above \(\beta\) are the solutions of \(f(x)=\beta\). Equivalently, we have \(x^{d} = \frac{\beta - t}{c}\). Choose one solution \(\gamma_{\beta}\) satisfying \(\gamma_{\beta}^{d} = \frac{\beta - t}{c}\). Since \(\mu_{d}\subset K\), the full fiber above \(\beta\) is \(\{\zeta \gamma_{\beta}:\zeta \in \mu_{d}\}\). Therefore any automorphism in \(\operatorname{Gal}(K_{n}/K_{n-1})\) acts on this fiber by multiplication by an element of \(\mu_{d}\). Hence
\[\operatorname{Gal}(K_{n}/K_{n-1})\hookrightarrow \prod_{\beta\in R_{n-1}}\mu_{d}\cong (C_{d})^{R_{n-1}}.\] An element of \(\mathcal{G}_{n}\) also permutes the fibers according to its restriction to \(\mathcal{G}_{n-1}\). Thus the action on the level-\(n\) preimage tree gives a natural homomorphism \[\mathcal{G}_{n}\longrightarrow (C_{d})^{R_{n-1}}\rtimes \mathcal{G}_{n-1}.\]
This homomorphism is injective because \(K_{n}\) is generated over \(K\) by the roots in \(R_{n}\). Hence \(\mathcal{G}_{n}\hookrightarrow (C_{d})^{R_{n-1}}\rtimes \mathcal G_{n-1}.\) Iterating this embedding gives
\( \mathcal{G}_{n}\hookrightarrow C_{d}\wr C_{d}\wr \dots \wr C_{d},
\) with \(n\) factors.
\end{proof}

\subsection{Capelli's Lemma}

The following irreducibility criterion will play a fundamental role in our analysis of the arboreal Galois groups associated with the families considered in Sections~4 and~5. It allows one to reduce the irreducibility of polynomial iterates to the irreducibility of certain polynomials over suitable field extensions; see \cite[Satz~4, p.~288]{Tschebotarow1950}.

\begin{lem}[Capelli's Lemma]
\label{Capelli's Lemma}
Let \(L\) be a field, let \(f(x)\in L[x]\) be irreducible, and let
\(\alpha\) be a root of \(f(x)\) in an algebraic closure of \(L\).
If \(g(x)\in L[x]\), then the composed polynomial \(f(g(x))\)
is irreducible over \(L\) if and only if \(g(x)-\alpha\)
is irreducible over \(L(\alpha)\).
\end{lem}

\subsection{Gekeler's Defect Criterion}
The comparison between arboreal and Drinfeld Galois representations in Section~6 requires a criterion of Gekeler describing the image of the adelic Galois representation attached to a twisted Carlitz module. We briefly recall the necessary definitions and results.

\subsubsection{Twisted Carlitz Modules}

The examples considered in Section~6 arise from twisted Carlitz modules. For \(\Delta\in K^\times\), the twisted Carlitz module associated to \(\Delta\) is the rank-one Drinfeld module defined by \(\rho_t^{\Delta}(x)=tx+\Delta x^q.\) \\
We refer the reader to \cite{GEKELER2016316} for the basic theory of twisted Carlitz modules and the associated adelic Galois representations. Following Gekeler, the failure of surjectivity of the Adelic image is measured by a numerical invariant called the \emph{defect}.

\subsubsection{Gekeler's Theorem}

We recall Gekeler's description of the image of the Adelic Galois representation attached to a twisted Carlitz module.\\
Let $\Delta=c^{k_0}Q_1^{k_1}\cdots Q_s^{k_s}$, where $c\in\mathbb \mu_d$ is a fixed primitive \((q-1)^{st}\) root of unity, the polynomials \(Q_1,\ldots,Q_s\in \mathcal{R}\) are distinct monic irreducible polynomials, and $$0\le k_0<q-1, \quad
0<k_i<q-1 \quad (1\le i\le s).$$ Set $D:=\deg(\Delta) =\sum_{i=1}^{s}k_i\deg(Q_i),$ and define 
\[
k_0^*=
\begin{cases}
k_0,
& \text{if $q$ or $D$ is even},\\[1mm]
k,
& \text{if $q$ and $D$ are odd},
\end{cases}
\]
where \(k\) is the unique integer satisfying
\[
k\equiv k_0+\frac{q-1}{2}\pmod{q-1},
\qquad
0\le k<q-1.
\]
Let \(N\in \mathcal{R} \) be a nonconstant polynomial. Suppose that \(Q_i\mid N \quad (1\le i\le r),\) and \(\gcd(Q_i,N)=1 \quad (r<i\le s).\) The following theorem, due to Gekeler, gives an explicit formula for the defect of the Galois representation attached to the twisted Carlitz module $\rho_t^{\Delta}(x)$.

\begin{thm}[\cite{GEKELER2016316}, Theorem 3.13]
\label{Gekeler}
The image of mod-\(N\) Galois representation associated to \(\rho_t^{\Delta}(x)\) has index in $(\mathbb F_q[t]/(N))^\times$ equal to
\[\operatorname{def}(\Delta,N)=\gcd\bigl(D-1,\,q-1,\,k_0^*,\,
k_{r+1},\ldots,k_s\bigr).\]
\end{thm}

\begin{cor}[\cite{GEKELER2016316}, Corollary 3.14]
\label{cor-gek}

As an immediate consequence, one obtains a corresponding description of the defect of the associated adelic Galois representation $\operatorname{def}\bigl(\rho^{(\Delta)}\bigr)
=
\gcd\bigl(
D-1,\,
q-1,\,
k_0^*
\bigr).$
\end{cor}

\section{A Kummer Independence Criterion}

\begin{prop}
\label{prop:general-kummer-independence}
Let \(f(x)=Ax^d+B\in K[x],\quad A,B\in K^\times.\)
For \(m\ge1\), set \(a_m:=f^{\circ m}(0).\)
Fix \(n\ge2\). Assume that \(f^{\circ m}(x)\) is separable over \(K\) for \(1\le m\le n\), and that \(\mathcal{G}_{n-1}\) acts transitively on \(R_{n-1}\). Suppose there exists a finite place \(P\) of \(K\) such that
\[
v_P(a_n)=1,
\qquad
v_P(a_m)=0\quad(1\le m<n),
\qquad
v_P(A)=v_P(B)=0.
\]
Then the Kummer classes \(\left[\frac{\beta-B}{A}\right],
\ \beta\in R_{n-1},\) are linearly independent in
\(K_{n-1}^\times/(K_{n-1}^\times)^d,\) i.e. if
\(\prod_{\beta\in R_{n-1}} \left(\frac{\beta-B}{A}\right)^{e_\beta}\in(K_{n-1}^{\times})^d\) for integers \(e_\beta\), then $e_\beta\equiv0\pmod d$ for every \(\beta\in R_{n-1}\). Consequently, we have \(\operatorname{Gal}(K_n/K_{n-1})\cong\prod_{\beta\in R_{n-1}} C_d=(C_d)^{R_{n-1}}.\)
\end{prop}

\begin{proof}

For each \(\beta\in R_{n-1}\), the equation
\(f(x)=\beta\) is equivalent to \( x^d=\frac{\beta-B}{A}.\) Then \(K_n/K_{n-1}\) is obtained by adjoining the \(d^{th}\) roots of the elements \(\frac{\beta-B}{A},\ \beta\in R_{n-1}.\) It remains to prove that their classes are independent in \(K_{n-1}^\times/(K_{n-1}^\times)^d\). Since \(v_P(a_n)=1\) then \(B\) is a root of \(f^{\circ(n-1)}(x)\) modulo \(P\). 

We show that $B$ root is simple. By the chain rule, \[\bigl(f^{\circ(n-1)}\bigr)'(x)=(Ad)^{n-1}
\prod_{i=0}^{n-2}f^{\circ i}(x)^{d-1}.\] Evaluating at \(x=B\), and using \(f^{\circ i}(B)=a_{i+1}\), we obtain \(v_P\left(\bigl(f^{\circ(n-1)}\bigr)'(B)\right)=0.\) Thus \(B\) is a simple root of \(f^{\circ(n-1)}(x)\) modulo \(P\).

Next we show that \(P\) is unramified in \(K_{n-1}/K\). Suppose that \(f^{\circ(n-1)}\) had a multiple root modulo \(P\). Then there would exist \(\gamma\) such that \( f^{\circ(n-1)}(\gamma)\equiv 0 \pmod P\) and \( \bigl(f^{\circ(n-1)}\bigr)'(\gamma)\equiv 0 \pmod P.\) Since \(v_P(A)=0\) and \(d\in \mathbb F_q^\times\), the derivative formula implies that \(f^{\circ i}(\gamma)\equiv 0 \pmod P\) for some \(0\le i\le n-2\). Applying composition with \(f^{\circ(n-1-i)}\) on both sides, we get \(a_{n-1-i}=f^{\circ(n-1-i)}(0)\equiv f^{\circ(n-1)}(\gamma)\equiv 0 \pmod P,\)
contradicting \(v_P(a_m)=0\ (1\le m<n).\) Hence \(P\nmid \operatorname{disc}(f^{\circ(n-1)})\), and therefore \(P\) is unramified in \(K_{n-1}\).

Since \(B\) is a simple root of \(f^{\circ(n-1)}(x)\) modulo \(P\), Hensel's lemma gives a unique root \(\widetilde{\beta}\in K_P\) of
\(f^{\circ(n-1)}(x)\) satisfying \(\widetilde{\beta}\equiv B \pmod P.
\) Choose an embedding \(K_{n-1}\hookrightarrow \overline{K_P}\) sending some root \(\beta^\ast\in R_{n-1}\) to \(\widetilde{\beta}\), and let \(\mathfrak p\) be the corresponding prime of \(K_{n-1}\) above \(P\). Then \(v_{\mathfrak p}(\beta^\ast-B)>0.\) No other root \(\beta\neq\beta^\ast\) can satisfy
\(v_{\mathfrak p}(\beta-B)>0\) because of the uniqueness of Hensel lifting. Hence \(v_{\mathfrak p}(\beta-B)\le 0 \ (\beta\neq\beta^\ast).
\) Using
\[
f^{\circ(n-1)}(B)
=
\operatorname{LC}(f^{\circ(n-1)})
\prod_{\beta\in R_{n-1}}(B-\beta),
\]
together with \(v_P(a_n)=1\), the fact that \(P\) is unramified in
\(K_{n-1}/K\), and the fact that the leading coefficient is a \(P\)-adic unit,
we obtain
\[
1
=
v_{\mathfrak p}\bigl(f^{\circ(n-1)}(B)\bigr)
=
\sum_{\beta\in R_{n-1}}v_{\mathfrak p}(B-\beta).
\]
Therefore we get
\[
v_{\mathfrak p}(\beta^\ast-B)=1,
\quad \textrm{ and }\quad
v_{\mathfrak p}(\beta-B)=0
\quad \textrm{ for }\beta\neq\beta^\ast.
\]
Since \(v_P(A)=0\), this gives \(v_{\mathfrak p}\left(\frac{\beta^\ast-B}{A}\right)=1,\) and \(v_{\mathfrak p}\left(\frac{\beta-B}{A}\right)=0 \ \textrm{ for $\beta\ne\beta^\ast.$}\) Now suppose that there is a Kummer relation \( \prod_{\beta\in R_{n-1}} \left(\frac{\beta-B}{A}\right)^{e_\beta}=w^d\) for some \(w\in K_{n-1}^\times\) and integers \(e_\beta\). Evaluating via \(v_{\mathfrak p}\), we obtain
\( e_{\beta^\ast}\equiv 0\pmod d.\) Let \(\beta_0\in R_{n-1}\) be arbitrary. Since \(\mathcal{G}_{n-1}\) acts transitively on
\(R_{n-1}\), choose \(\sigma\in \mathcal G_{n-1}\) such that \(
\sigma(\beta^\ast)=\beta_0.\) Evaluating by the conjugate valuation \(\sigma(\mathfrak p)\) gives \(e_{\beta_0}\equiv 0\pmod d.\) Since \(\beta_0\) was arbitrary, all \(e_\beta\) are divisible by \(d\). Therefore the classes \(
\left[
\frac{\beta-B}{A}
\right],
\
\beta\in R_{n-1},
\)
are linearly independent in \(K_{n-1}^\times/(K_{n-1}^\times)^d\). Hence by the multi-radical form of Kummer theory, this linear independence implies
\[
\operatorname{Gal}(K_n/K_{n-1})
\cong
\prod_{\beta\in R_{n-1}} C_d.
\]
\end{proof}

\section{Arboreal Maximality for $f_M(x)=t^M x^{q-1}+t$} \label{Sec.AMf_m}

In this section, under the assumption that $q\geq 5$ is an odd prime power and $d=q-1$, we establish a primitive simple divisor theorem for the family
\[f_M(x)=t^M x^{d}+t,\]
where \(M\ge 0\). Note that the additional assumption on $q$ is used in the proof of Lemma~\ref{lem:primitive-radical-estimate-tM}. Then we give a sufficient condition for $M$ to satisfy Kummer independence criterion and maximality of first-level Galois group, the two properties ensure maximality of arboreal Galois image.

\subsection{Primitive Simple Divisors}\label{psd1}

 The existence of primitive simple divisors in the critical orbit of \(f_M\) will provide the key input needed for the Kummer independence criterion of Section $3$. For \(n\ge 1\), define
\(a_n=f_M^{\circ(n)}(0).\)

\begin{thm}
\label{thm:PSD}
Assume that \(M+d-1 \not\equiv 0 \pmod p.\) Then for every \(n\ge 1\) there exists an irreducible polynomial \(P_n\in \mathcal{R} \) such that \[v_{P_n}(a_n)=1, \qquad P_n\nmid a_1a_2\cdots a_{n-1}.
\] \\ In particular, every element \(a_n\) of the critical orbit has a primitive simple divisor.
\end{thm}

The proof of Theorem \ref{thm:PSD} proceeds in three steps. We first determine the valuation of the critical orbit at the place \(t\) and obtain a normalized recursion. We then establish a lower bound for the primitive radical of the normalized orbit sequence using Theorem $ \ref{thm:mason-stothers}.$  Finally, a degree-counting argument yields the existence of a primitive simple divisor.

\begin{lem}
\label{lem:valuation-recursion-tM}
For every \(n\ge 1\), we have 
\(v_t(a_n)=1\). 
\end{lem}

\begin{proof}
We proceed by induction on \(n\). For \(n=1\), we have
\(a_1=f_M(0)=t,\) and therefore \(v_t(a_1)=1.\) Assume that \(v_t(a_n)=1\) for some \(n\ge1\). Then we may write \(a_n=tb_n,\) where \(b_n\in \mathcal{R}\) and \(t\nmid b_n\), we obtain
\(a_{n+1}
=f_M(a_n)
=t^M a_n^{\,d}+t.
\)
Substituting \(a_n=tb_n\) gives
\(
a_{n+1}
=t^M(tb_n)^{d}+t
=t\Bigl(t^{M+d-1}b_n^{\,d}+1\Bigr).
\)
Since \(M+d-1\ge0\), the polynomial
\(
t^{M+d-1}b_n^{\,d}+1
\)
is congruent to \(1\) modulo \(t\). Hence it is not divisible by \(t\). 
\end{proof}

Defining \(b_1=1\),
\(a_n=tb_n, \) and set $E:=M+d-1$ the above computation shows that
\[
b_n=t^{E}b_{n-1}^{\,d}+1
\qquad (n\ge2).
\]

\begin{lem}
\label{lem:primitive-radical-estimate-tM}
Let \(N_n\) denote the degree of the primitive radical of \(b_n\) and \(B_n:=\deg\operatorname{rad}(b_n).\) Assume \(M+d-1\not\equiv 0 \pmod p\), then
\(
N_n>\frac{D_n}{2}
\)
for every \(n\ge 2\).
\end{lem}

\begin{proof}
Set \(D_n=\deg b_n.\) Then \(D_1=0\) and, for \(n\ge2\), \( D_n=E+dD_{n-1}.\) Hence \[
D_n=E(1+d+\cdots+d^{n-2})
=E\frac{d^{n-1}-1}{d-1}.
\]
Since \(b_n-t^{E}b_{n-1}^{d}=1,\) the polynomials
\(b_n,\ -t^{E}b_{n-1}^{d},\ -1\) are pairwise coprime. Moreover, \(p\nmid E\) by hypothesis, so \(t^{E}b_{n-1}^{d}\) is not a \(p^{th}\) power in \(R\). Therefore the exceptional case of Theorem~\ref{thm:mason-stothers} does not occur, and Theorem~\ref{thm:mason-stothers} yields \(D_n\le\deg\operatorname{rad}(t^E b_{n-1}^d b_n)-1.\) Thus \(D_n\le1+B_{n-1}+B_n-1=B_{n-1}+B_n.\) But \(B_{n-1}\le D_{n-1}\), we obtain \(B_n\ge D_n-D_{n-1}.\) Using the formula for \(D_n\), we get \(D_n-D_{n-1}=Ed^{n-2}.\) Hence \(B_n\ge Ed^{n-2}.\) Suppose that an irreducible polynomial \(P\ne t\) divides both \(b_n\) and some earlier \(b_m\), with \(1\le m<n\). Then \(P\mid a_n\) and \(P\mid a_m\). Hence, modulo \(P\), the orbit of \(0\) under \(f_M\) returns to \(0\) at both times \(m\) and \(n\). Let \(r\) be the first positive return time. Then all return times are multiples of \(r\), so \( r\mid m\ \text{and}\ r\mid n.\) In particular, \(r<n\), \(r\mid n\), and \(P\mid a_r\). Since \(P\ne t\), this implies \(P\mid b_r\). Therefore every nonprimitive irreducible divisor of \(b_n\) divides \(\prod_{\substack{r\mid n\\ r<n}} b_r.\) Consequently, \(N_n \ge B_n-\sum_{\substack{r\mid n\\ r<n}}D_r.\)
Using \(B_n\ge Ed^{n-2}\), we obtain
\(N_n\ge Ed^{n-2}-\sum_{\substack{r\mid n\\ r<n}}D_r.
\) For \(n=2\) and \(n=3\), the only proper divisor contribution is \(D_1=0\). Hence \(N_n\ge Ed^{n-2}\) and the desired inequality holds for \(n=2,3\).\\ Now assume \(n\ge4\). Every proper divisor \(r<n\) of \(n\) satisfies
\(r\le \frac n2.\)
We get
\[
\sum_{\substack{r\mid n\\ r<n}}D_r
\le
\sum_{r=1}^{\lfloor n/2\rfloor}D_r=\frac{E}{d-1}\sum_{r=1}^{\lfloor n/2\rfloor}d^{r-1}-1<
\frac{E}{d-1}\sum_{r=1}^{\lfloor n/2\rfloor}d^{r-1}
<Ed^{\lfloor n/2\rfloor}\frac{1}{(d-1)^2}.
\]
Since \(n\ge4\), we apply \(\lfloor n/2\rfloor\le n-2\) to the above inequality to get
\[
\sum_{r=1}^{\lfloor n/2\rfloor}D_r
<Ed^{\lfloor n/2\rfloor}\frac{1}{(d-1)^2}<
Ed^{n-2}
-
\frac{Ed^{n-1}}{2(d-1)}.
\]
Note that the second inequality above used the assumption $d=q-1\geq 4$. 

It follows that
\[
Ed^{n-2}
-
\sum_{r=1}^{\lfloor n/2\rfloor}D_r
>
\frac{Ed^{n-1}}{2(d-1)}
>
\frac{E(d^{n-1}-1)}{2(d-1)}
=
\frac{D_n}{2}.
\]
Thus we get
\[
N_n
\ge
Ed^{n-2}
-
\sum_{\substack{r\mid n\\ r<n}}D_r
>
\frac{D_n}{2}.
\]
\end{proof}

\begin{proof}[Proof of Theorem\ref{thm:PSD}]
For \(n=1\), we have \(a_1=t,\) so \(a_1\) possesses a primitive simple divisor. Now let \(n\ge2\). By Lemma~\ref{lem:primitive-radical-estimate-tM}, \(N_n>\frac{D_n}{2}.\) Suppose, for contradiction, that every primitive irreducible divisor of \(b_n\) occurs with multiplicity at least \(2\). Then the primitive part of \(b_n\) has degree at least \(2N_n>D_n=\deg(b_n),\) which is impossible. Hence there exists a primitive irreducible divisor \(P_n\mid b_n\) such that \(v_{P_n}(b_n)=1.\) By Lemma~\ref{lem:valuation-recursion-tM}, \(a_n=tb_n.\) Since \(P_n\neq t\), we obtain \(v_{P_n}(a_n)=1.\) Moreover, the primitivity of \(P_n\) for \(b_n\) implies that \(P_n\nmid a_1a_2\cdots a_{n-1}.\) Therefore \(P_n\) is a primitive simple divisor of \(a_n\).
\end{proof}

\begin{cor}
\label{cor:monicPSD}
Let \(f_0(x)=x^{q-1}+t.\) Then every element of \(a_n=f_0^{\circ n}(0)\) has a primitive simple divisor.
\end{cor}

\begin{proof}
This is the case \(M=0\) of Theorem~\ref{thm:PSD}. Since
\(0+(q-1)-1=q-2\not\equiv0\pmod p,\) the hypothesis of Theorem~\ref{thm:PSD} is satisfied.
\end{proof}

\begin{rem}
Note that \(M=1\), namely $f_1(x)=tx^{q-1}+t$, is the first genuinely nonconstant example covered by Theorem~\ref{thm:PSD}. It shows that the primitive simple divisor phenomenon is not restricted to monic polynomials.
\end{rem}

\subsection{Kummer Independence and Full Cyclic Wreath Products}

 \begin{lem}
 \label{Lemma:f_M_separable}
For every \(n\geq 1\), the polynomial \(f_M^{\circ n}(x)\) is separable over \(K\).
\end{lem}

\begin{proof}
Since \(f_M'(x)=t^M d x^{d-1}\) and \(p\nmid d\), the only critical point of \(f_M\) is \(0\).
Suppose that \(f_M^{\circ n}(x)\) has a multiple root \(\alpha\). Then
\(f_M^{\circ n}(\alpha)=0\) and \((f_M^{\circ n})'(\alpha)=0\). By the chain rule,
\[
(f_M^{\circ n})'(x)
=
\prod_{i=0}^{n-1} f_M'\bigl(f_M^{\circ i}(x)\bigr).
\]
Hence \(f_M^{\circ i}(\alpha)=0\) for some \(0\leq i\leq n-1\). Applying
\(f_M^{\circ(n-i)}\) gives \(f_M^{\circ(n-i)}(0)=0,\)
contradicting Lemma~\ref{lem:valuation-recursion-tM}, which gives
\(v_t(f_M^{\circ m}(0))=1\) for every \(m\geq 1\). Therefore
\(f_M^{\circ n}(x)\) has no multiple roots, and hence is separable over \(K\).
\end{proof}

\begin{thm}\label{thm:conditional-wreath}
Let $q=p^e\geq 5$ be an odd prime power, $d=q-1$, and $K=\mathbb F_q(t)$. Assume that \(M+d-1\not\equiv 0 \pmod p\) and that the Kummer class \([-t^{1-M}]\in K^\times/(K^\times)^d\) has order \(d\). Then for every \(n\ge 1\), the n-th iterated Galois group of $f_M(x)=t^Mx^d+t$ satisfies \( \mathcal{G}_n \cong C_d \wr C_d \wr \cdots \wr C_d,\) with \(n\) many factors.
\end{thm}

\begin{proof}
We proceed by induction on \(n\). For \(n=1\), the equation
\(f_M(x)=0\) is equivalent to \(x^d=-t^{1-M}.\) Since \(\mu_d\subset K\) and the class \([-t^{1-M}]\) has order \(d\) in \(K^\times/(K^\times)^d\), Theorem~\ref{Kummer-Theory} implies that \(\mathcal{G}_1\cong C_d.\) In particular, \(\mathcal G_1\) acts transitively on \(R_1\). Now assume that \(\mathcal{G}_{n-1}\cong C_d\wr C_d\wr\cdots\wr C_d\) with \(n-1\) factors. Then \(\mathcal{G}_{n-1}\) acts transitively on \(R_{n-1}\). By Theorem \ref{thm:PSD}, \(a_n\) has a primitive simple divisor \(P_n\). Hence \(
v_{P_n}(a_n)=1, \ v_{P_n}(a_m)=0 \ (1\le m<n), \)
and \( v_{P_n}(t)=v_{P_n}(t^M)=0.\) Together with Lemma \ref{Lemma:f_M_separable}, all conditions in Proposition~\ref{prop:general-kummer-independence}
are satisfied. It follows that the Kummer classes
\( \left[ \frac{\beta-t}{t^M} \right],
\textrm{ where } \beta\in R_{n-1}, \) are linearly independent in
\(K_{n-1}^\times /(K_{n-1}^\times)^d. \) Hence Proposition~\ref{prop:general-kummer-independence} yields
\(\operatorname{Gal}(K_n/K_{n-1})\cong\prod_{\beta\in R_{n-1}} C_d.
\) Since \(\mathcal G_{n-1}\) acts on \(R_{n-1}\), the natural action on the preimage tree yields \[
\mathcal{G}_n
\cong
\left(
\prod_{\beta\in R_{n-1}} C_d
\right)
\rtimes \mathcal{G}_{n-1}.
\]
Therefore \( \mathcal{G}_n \cong C_d \wr \mathcal{G}_{n-1}.\) Using the induction hypothesis, we obtain \( \mathcal{G}_n\cong C_d\wr C_d\wr\cdots\wr C_d,\) with \(n\) factors.
\end{proof}

Theorem~\ref{thm:conditional-wreath} establishes the maximality of the arboreal Galois groups associated with the family \(f_M(x)=t^Mx^d+t\). The following remarks clarify the hypotheses of the theorem and explain the role of the primitive simple divisor theorem in the proof.

\begin{rem}\label{maxcond}
The first-level hypothesis in Theorem~\ref{thm:conditional-wreath} can be checked explicitly. Since
\(
[-t^{1-M}]
=
[-1]\,[t^{1-M}]
\)
in \(K^\times/(K^\times)^d\), the class \([t^{1-M}]\) has order
\(
\frac{d}{\gcd(d,1-M)}.
\)
Consequently,
\[
\operatorname{ord}\bigl([-t^{1-M}]\bigr)
=
\operatorname{lcm}\left(
2,
\frac{d}{\gcd(d,1-M)}
\right).
\]
Hence \([-t^{1-M}]\) has order \(d\) if and only if
\[
\operatorname{lcm}\left(
2,
\frac{d}{\gcd(d,1-M)}
\right)
=d.
\]
\end{rem}

\begin{rem}
The transitivity hypothesis in Theorem \ref{thm:conditional-wreath} is essential. If the first-level class \([-t^{1-M}] \in K^{\times}/(K^{\times})^{d}\) does not have order \(d\), then \(\mathcal G_{1}\) is not transitive on \(R_{1}\), and the full cyclic wreath product need not occur.\\
For \(\beta \in R_{n-1}\), set
\[
u_{\beta} := \frac{\beta - t}{t^{M}} \qquad and \qquad 
\Gamma_{n} := \left\langle [u_{\beta}] : \beta \in R_{n-1} \right\rangle
\subset K_{n-1}^{\times}/(K_{n-1}^{\times})^{d}.
\]
Since \(K_{n}\) is obtained from \(K_{n-1}\) by adjoining \(d^{th}\) roots of the elements \(u_{\beta}\), Theorem \ref{Kummer-Theory} yields
\[
\operatorname{Gal}(K_{n}/K_{n-1}) \cong \operatorname{Hom}(\Gamma_{n}, \mu_{d}).
\]
Thus, in the nontransitive case, the correct replacement for the full wreath product description is the Kummer-theoretic description above together with the natural embedding \(\mathcal{G}_{n} \hookrightarrow (C_{d})^{R_{n-1}} \rtimes \mathcal{G}_{n-1}.\)
The subgroup \(\Gamma_{n}\) records precisely the Kummer relations among the classes
\(\left[\frac{\beta - t}{t^{M}}\right], \ \beta \in R_{n-1}.
\)
\end{rem}

\begin{rem}
The role of the primitive simple divisor theorem is to provide a valuation separating one Kummer class \(\left[\frac{\beta^{\ast} - t}{t^{M}}\right].\) When \(\mathcal{G}_{n-1}\) acts transitively on \(R_{n-1}\), this valuation may be transported by Galois conjugation to every root \(\beta \in R_{n-1}\). Consequently, the classes
\(\left[\frac{\beta - t}{t^{M}}\right], \ \beta \in R_{n-1},
\) are linearly independent in \(K_{n-1}^{\times} / (K_{n-1}^{\times})^{d}.\)  If transitivity fails, the same argument only yields independence on the \( \mathcal{G}_{n-1}\)-orbit containing \(\beta^{\ast}\), additional orbitwise information is still required.
\end{rem}

\section{The Irreducible Conjugate Case}\label{Sec.Irreducible Conjugate} 

Let \(f_c(x)=c x^d+t\in K[x]. \) Assume throughout this section that $q\geq 7$, so $d=q-1 \geq 6$. This additional assumption on $q$ is used in the proof of Proposition~\ref{lem:primitive_place_F}. We further assume that \(c\in (K^\times)^{d-1}\), and
\(v_t(c)=0\). Choose \(u\in K^\times\) satisfying \(u^{d-1}=c^{-1},\) we define
\(F(x)=u^{-1}f_c(ux) =x^d+C, \textrm{ with } \ C=\frac{t}{u}.\) Indeed,
\[
u^{-1}f_c(ux)
=
u^{-1}\left(cu^d x^d+t\right)
=
cu^{d-1}x^d+\frac{t}{u}
=
x^d+\frac{t}{u}.
\]
From
\(f_c(x)=uF(u^{-1}x)\), one can deduce that
\(f_c^{\n}(x)=uF^{\n}(u^{-1}x)\) for every \(n\ge 1\). Therefore \(f_c\) and \(F\) are linearly conjugate over \(K\) forces them to share the same splitting field over \(K\). Namely, \(\operatorname{Spl}_K(f_c^{\circ n})
=\operatorname{Spl}_K(F^{\circ n}).\) Thus the corresponding iterated Galois groups coincide.  The irreducibility and Galois-theoretic properties of the iterates of \(f_c\) may be studied through the simpler polynomial \( F(x)=x^d+C.\)

\subsection{Irreducibility of the Iterates}
 
\begin{lem}\label{lem:critical-orbit}
For every integer \(n\ge 1\), \(v_t\!\left(F^{\circ n}(0)\right)=1.\) In particular, \(F^{\circ n}(0)\neq 0\) and \(F^{\circ n}(x)\) is separable over \(K\).
\end{lem}
\begin{proof} Since \(C=t/u\) and \(v_t(u)=0\), we have \(
v_t(C)=1.\) We prove by induction that \(v_t\!\left(F^{\circ n}(0)\right)=1\) for every \(n\ge 1\). For \(n=1\), we have \(F(0)=C,\) so \(v_t(F(0))=v_t(C)=1.\) Assume now that \(v_t\!\left(F^{\circ n}(0)\right)=1.\) Then \(F^{\circ(n+1)}(0)=F\!\left(F^{\circ n}(0)\right)=\left(F^{\circ n}(0)\right)^d+C.\) Since \(d\ge 2\), we have \(v_t\!\left(\left(F^{\circ n}(0)\right)^d\right)=d>1,\) whereas \(v_t(C)=1.\) Therefore \(v_t\!\left(F^{\circ (n+1)}(0)\right)=1.\) In particular, \(F^{\circ n}(0)\neq 0\) for every \(n\ge 1\). It remains to prove separability. Since \(F'(x)=d x^{d-1},\) the only critical point of \(F\) is \(0\). Suppose that \(F^{\circ n}(x)\) had a repeated root \(\alpha\). Then \(F^{\circ n}(\alpha)=0\) and \(\left(F^{\circ n}\right)'(\alpha)=0.\) By the chain rule,
\[
\left(F^{\circ n}\right)'(x)
=
d^n\prod_{i=0}^{n-1}
\left(F^{\circ i}(x)\right)^{d-1}.
\]
Hence \(F^{\circ i}(\alpha)=0\) for some \(0\le i\le n-1\). Applying \(F^{\circ(n-i)}\) gives \(F^{\circ(n-i)}(0)=0,\) contradicting the valuation statement proved above. Therefore \(F^{\circ n}(x)\) is separable over \(K\).
\end{proof}

\begin{thm}\label{thm:F-iterates-irreducible}
For every integer \(n\ge 1\), the iterate \(F^{\circ n}(x)
\) is irreducible over \(K\).
\end{thm}
\begin{proof}
We prove the result by induction on \(n\). For \(n=1\), we compute the valuation $v_t(-C)=1$. Applying Theorem \ref{Kummer-Theory}, we know that the class $[-C]\in K^\times/(K^\times)^d$ has order $d$, thus we can deduce that \(F(x)\) is irreducible over \(K\). Now assume that \(F^{\circ(n-1)}(x)\) is irreducible over \(K\), where \(n\ge 2\). Let \(\alpha\) be a root of \(F^{\circ (n-1)}(x)\). By Lemma \ref{Capelli's Lemma} , \(F^{\circ n}(x)\) is irreducible over \(K\) if and only if \(F(x)-\alpha\) is irreducible over \(K(\alpha)\). Since \(F(x)-\alpha=x^d+C-\alpha=x^d-(\alpha-C),\) and since \(\mu_d\subset K\subset K(\alpha)\), Theorem \ref{Kummer-Theory} shows that \(F(x)-\alpha\) is irreducible over \(K(\alpha)\) if and only if \(\alpha-C\notin \left(K(\alpha)^\times\right)^\ell\) for every prime divisor \(\ell\mid d\). We prove this using the norm. Since \(F^{\circ (n-1)}(x)\) is monic and irreducible over \(K\), we have \[N_{K(\alpha)/K}(\alpha-C)=\prod_{\alpha'}(\alpha'-C)=\pm F^{\circ(n-1)}(C),\] where the product runs over the conjugates \(\alpha'\) of \(\alpha\). But \(C=F(0)\), so \( F^{\circ (n-1)}(C)=F^{\circ n}(0).\) Hence \(N_{K(\alpha)/K}(\alpha-C)=\pm F^{\circ n}(0).\) By Lemma~\ref{lem:critical-orbit}, \( v_t\!\left(F^{\circ n}(0)\right)=1.\) Therefore \(v_t\!\left(N_{K(\alpha)/K}(\alpha-C)\right)=1.\) Now fix a prime divisor \(\ell\mid d\). If \( \alpha-C\in \left(K(\alpha)^\times\right)^\ell,\) then its norm would lie in \((K^\times)^\ell\). Hence every valuation of the norm would be divisible by \(\ell\). This contradicts \( v_t\!\left(N_{K(\alpha)/K}(\alpha-C)\right)=1.\) Therefore \(\alpha-C\notin \left(K(\alpha)^\times\right)^\ell\) for every prime divisor \(\ell\mid d\). Thus \(F(x)-\alpha\) is irreducible over \(K(\alpha)\). By Lemma \ref{Capelli's Lemma} , \(F^{\circ n}(x)\) is irreducible over \(K\). This completes the induction.
\end{proof}

\begin{cor}\label{cor:f-iterates-irreducible}
For every integer \(n\ge 1\), the iterate \(f_c^{\circ n}(x)\) is irreducible and separable over \(K\).
\end{cor}
\begin{proof}
By Theorem~\ref{thm:F-iterates-irreducible}, \(F^{\circ n}(x)\) is irreducible and separable over \(K\). Since \(f_c^{\circ n}(x)=u\,F^{\circ n}(u^{-1}x),\) the polynomials \(f_c^{\circ n}\) and \(F^{\circ n}\) are linearly conjugate over \(K\). Hence irreducibility and separability are preserved.
\end{proof}

\subsection{Primitive Simple Divisors and Kummer Independence}\label{psd2}
Having established the irreducibility of all iterates of \(F\), we now turn to the existence of primitive simple divisors in its critical orbit. The main result of this subsection shows that, for every \(n\ge1\), the element $F^{\circ n}(0)$ has a primitive simple divisor. This provides the key arithmetic input needed for the Kummer independence criterion of Section~\ref{subsec:full_wreath_product}.

\begin{prop}\label{lem:primitive_place_F}
For every integer $n \ge 1$, there exists a finite place $P$ of $K$ such that \[v_P\left(F^{\circ (n+1)}(0)\right)=1,\qquad v_P\left(F^{\circ m}(0)\right)=0 \quad (1\le m\le n).\]
In particular, \(\gcd\left(v_P\left(F^{\circ(n+1)}(0)\right),d\right)=1.\)
\end{prop}

\begin{proof}
Write \(C=A/B\), where \(A,B\in \mathcal R\) are coprime. Since \(v_t(C)=1\), we choose \(A\) and \(B\) so that
\begin{equation}\label{eq:AB-t-conditions}
v_t(A)=1,\qquad t\nmid B.
\end{equation}
Set \(M_A=\deg A\), \(M_B=\deg B\), and \(M=\max\{M_A,M_B\}\).
For \(m\geq 1\), write \(F^{\m}(0)=C\,G_m(C),\)
where
\begin{equation}\label{eq:G-recursion}
G_1(x)=1,\qquad G_{m+1}(x)=x^{d-1}G_m(x)^d+1.
\end{equation}
Then \(\deg_x G_m=d^{m-1}-1.\) Define \(H_m\in \mathcal R\) by
\[G_m(C)=\frac{H_m}{B^{d^{m-1}-1}}.\]
Then \(H_1=1\), and from \eqref{eq:G-recursion} we obtain
\begin{equation}\label{eq:H-recursion}
H_{m+1}=A^{d-1}H_m^d+B^{d^m-1}.
\end{equation}
Consequently,
\begin{equation}\label{eq:Fm-Hm}
F^{\m}(0)=\frac{A H_m}{B^{d^{m-1}}}.
\end{equation}\\
We first record the elementary coprimality properties that will be used repeatedly. From
\(\gcd(A,B)=1\) and \eqref{eq:H-recursion}, one obtains
\begin{equation}\label{eq:H-coprime-AB}
\gcd(H_m,A)=\gcd(H_m,B)=1
\qquad (m\geq 1).
\end{equation}
Moreover, since \(H_{m+1}\equiv B^{d^m-1}\pmod{H_m},\) we have
\begin{equation}\label{eq:H-coprime-previous}
\gcd(H_{m+1},H_m)=1.
\end{equation}
Finally, \eqref{eq:AB-t-conditions} and \eqref{eq:H-recursion} imply by induction that
\begin{equation}\label{eq:t-not-divide-Hm}
t\nmid H_m
\qquad (m\geq 1).
\end{equation}\\
We also need a direct degree estimate for \(H_m\). Since \(G_m(X)\) has degree \(d^{m-1}-1\), the numerator of \(G_m(A/B)\) has degree at most \((d^{m-1}-1)M\). Hence
\begin{equation}\label{eq:degree-Hm-upper}
\deg H_m\leq (d^{m-1}-1)M.
\end{equation}\\
We now treat first the case \(n=1\). In this case
\(H_2=A^{d-1}+B^{d-1}.\) We claim that \(H_2\) has an irreducible divisor \(P\) such that \(P\nmid AB,\ v_P(H_2)=1.\) Indeed, suppose for contradiction that every irreducible divisor of \(H_2\) outside \(AB\) occurs with multiplicity at least two. Since \(\gcd(H_2,AB)=1\), this implies \(\operatorname{rad}(H_2)\mid H_2'.\) Let \(P\mid H_2\). Since \(P\nmid AB\), we have
\(A^{d-1}\equiv -B^{d-1}\pmod P.\) Using \(H_2'=(d-1)\left(A^{d-2}A'+B^{d-2}B'\right),\)
and multiplying by \(AB\), we obtain \(P\mid AB'-A'B.\) Thus
\begin{equation}\label{eq:rad-H2-Wronskian}
\operatorname{rad}(H_2)\mid AB'-A'B.
\end{equation}
The polynomial \(AB'-A'B\) is nonzero. Otherwise, \((A/B)'=0\), so \(A/B\in \mathbb{F}_q(t^p)\), contradicting \(v_t(A/B)=1\). Therefore
\begin{equation}\label{eq:Wronskian-degree}
\deg \operatorname{rad}(H_2)\leq \deg(AB'-A'B)\leq \deg A+\deg B-1\leq 2M-1.
\end{equation}
On the other hand, applying Theorem \ref{thm:mason-stothers} to \(A^{d-1}+B^{d-1}-H_2=0\) gives \((d-1)M<\deg\operatorname{rad}(ABH_2).\)
Since \(\deg\operatorname{rad}(AB)\leq 2M\), we get
\begin{equation}\label{eq:rad-H2-lower}
\deg\operatorname{rad}(H_2)>(d-3)M.
\end{equation}
In our setting \(d=q-1\geq 6\), hence \eqref{eq:rad-H2-lower} contradicts to \eqref{eq:Wronskian-degree}. Hence such a divisor \(P\) exists. By \eqref{eq:Fm-Hm}, we then have \[v_P(F^{\circ 2}(0))=1,\ v_P(F(0))=0.\] 
Now assume \(n\geq 2\). Let \(S_{n+1}\) be the product of the distinct irreducible divisors of \(H_{n+1}\) which do not divide \(ABH_1H_2\cdots H_n.\) We apply Theorem \ref{thm:mason-stothers} to
\begin{equation}\label{eq:Mason-main-relation}
A^{d-1}H_n^d+B^{d^n-1}-H_{n+1}=0.
\end{equation}
The three terms are pairwise coprime by \eqref{eq:H-coprime-AB} and \eqref{eq:H-coprime-previous}. Moreover, they are not all \(p^{th}\) powers, since \(v_t(A^{d-1}H_n^d)=d-1\) by \eqref{eq:t-not-divide-Hm}, and \(p\nmid d-1\). Therefore \ref{eq:Mason-main-relation} gives \[
\max\{\deg(A^{d-1}H_n^d),\deg(B^{d^n-1}),\deg H_{n+1}\}
<
\deg\operatorname{rad}(ABH_nH_{n+1}).
\]

We claim that \[\max\{\deg(A^{d-1}H_n^d),\deg(B^{d^n-1}),\deg H_{n+1}\} \ge (d^n-1)M.\]
If \(M_B=M\), this follows from $\deg(B^{d^n-1})=(d^n-1)M.$ If \(M_A=M>M_B\), we prove by induction that $\deg H_{m+1}=(d^m-1)M \quad (m\ge 1).$ For \(m=1\), $H_2=A^{d-1}+B^{d-1},$ and since \(\deg A>\deg B\), no cancellation occurs, so $\deg H_2=(d-1)M.$ Assume the claim holds for \(m-1\), namely $\deg H_m=(d^{m-1}-1)M.$ Then
\[\deg(A^{d-1}H_m^d)=(d-1)M+d(d^{m-1}-1)M =(d^m-1)M,\]
whereas $\deg(B^{d^m-1})=(d^m-1)M_B<(d^m-1)M.$ Thus no cancellation occurs in $H_{m+1}=A^{d-1}H_m^d+B^{d^m-1},$ and hence $\deg H_{m+1}=(d^m-1)M.$ Therefore the claim follows in both cases. Hence
\begin{equation}\label{eq:rad-Hnplus1-lower-1}
\deg\operatorname{rad}(H_{n+1})
>
(d^n-1)M-\deg\operatorname{rad}(ABH_n).
\end{equation}
Using \eqref{eq:degree-Hm-upper}, we get
\[
\deg\operatorname{rad}(ABH_n)\leq 2M+\deg H_n
\leq 2M+(d^{n-1}-1)M.
\]
Thus
\begin{equation}\label{eq:rad-Hnplus1-lower-2}
\deg\operatorname{rad}(H_{n+1})
>
\bigl((d-1)d^{n-1}-2\bigr)M.
\end{equation}
Therefore
\(
\deg S_{n+1}
\geq
\deg\operatorname{rad}(H_{n+1})-\deg\operatorname{rad}(ABH_1\cdots H_n).
\)
Using \eqref{eq:degree-Hm-upper}, this gives
\begin{equation}\label{eq:Sn-lower}
\deg S_{n+1}
>
\left(
(d-1)d^{n-1}
-\frac{d^n-1}{d-1}
+n-4
\right)M.
\end{equation}
On the other hand,
\begin{equation}\label{eq:2M-Hn-upper}
2M+\deg H_n\leq (d^{n-1}+1)M.
\end{equation}
But since \(d\geq 6\), the estimate
\[
(d-1)d^{n-1}
-\frac{d^n-1}{d-1}
+n-4
>
d^{n-1}+1
\]
holds for every \(n\geq 2\). Hence \eqref{eq:Sn-lower} and \eqref{eq:2M-Hn-upper} imply
\begin{equation}\label{eq:Sn-large}
\deg S_{n+1}>2M+\deg H_n.
\end{equation}
We now prove that some irreducible divisor of \(S_{n+1}\) appears in \(H_{n+1}\) with multiplicity one. Suppose not. Then every irreducible divisor of \(S_{n+1}\) divides \(H_{n+1}'\), and hence \(S_{n+1}\mid H_{n+1}'.\) Differentiating \eqref{eq:H-recursion}, we obtain \(H_{n+1}'=A^{d-2}H_n^{d-1}\bigl((d-1)A'H_n+dAH_n'\bigr)+(d^n-1)B^{d^n-2}B'.\) Let \(P\mid S_{n+1}\). Then \(P\nmid ABH_n\), and since \(P\mid H_{n+1}\), we have \(A^{d-1}H_n^d\equiv -B^{d^n-1}\pmod P.\)
Using also \(P\mid H_{n+1}'\), we obtain \(P\mid J_n,\) where
\begin{equation}\label{eq:Jn-definition}
J_n=B\bigl((d-1)A'H_n+dAH_n'\bigr)
-
(d^n-1)AH_nB'.
\end{equation}
Since \(S_{n+1}\) is squarefree, this gives
\begin{equation}\label{eq:Sn-divides-Jn}
S_{n+1}\mid J_n.
\end{equation}
We claim that \(J_n\neq 0\). If \(J_n=0\), then
\(B(A^{d-1}H_n^d)'-(d^n-1)A^{d-1}H_n^dB'=0.\)
Equivalently,
\[\left(\frac{A^{d-1}H_n^d}{B^{d^n-1}}\right)'=0.
\]
Thus \(A^{d-1}H_n^d/B^{d^n-1}\in \mathbb{F}_q(t^p)\). But by \eqref{eq:AB-t-conditions} and \eqref{eq:t-not-divide-Hm},
\[
v_t\left(\frac{A^{d-1}H_n^d}{B^{d^n-1}}\right)=d-1,
\]
which is impossible because every valuation of an element of \(\mathbb{F}_q(t^p)\) is divisible by \(p\), whereas \(p\nmid d-1\). Hence \(J_n\neq 0\).\\
Moreover, from \eqref{eq:Jn-definition},
\begin{equation}\label{eq:Jn-degree}
\deg J_n\leq \deg A+\deg B+\deg H_n-1\leq 2M+\deg H_n-1.
\end{equation}
Now \eqref{eq:Sn-large}, \eqref{eq:Sn-divides-Jn}, and \eqref{eq:Jn-degree} give a contradiction. Therefore there exists an irreducible divisor \(P\mid S_{n+1}\) such that \(v_P(H_{n+1})=1.\) Since \(P\nmid ABH_1\cdots H_n\), equation \eqref{eq:Fm-Hm} gives \(v_P(F^{\circ (n+1)}(0))=1,\textrm{ and } v_P(F^{\m}(0))=0 \quad \textrm{ for }1\leq m\leq n.\)
\end{proof}

\subsection{The Full Cyclic Wreath Product}\label{subsec:full_wreath_product}

\begin{thm}
\label{thm:conjugate-full-wreath}
Let \(q=p^e\ge 7\) be an odd prime power,  $K=\mathbb F_q(t)$, \(d=q-1\), and let
\[
f_c(x)=cx^d+t\in K[x].
\]
Assume that
\(c\in (K^\times)^{d-1},
\textrm{ and }
v_t(c)=0.\)
Choose \(u\in K^\times\) such that \(u^{d-1}=c^{-1}\), and set
\[
F(x)=u^{-1}f_c(ux)=x^d+C,
\quad \textrm{ with }
C=\frac{t}{u}.
\]
Denote 
\(
K_n=\operatorname{Spl}_K(F^{\circ n})=\operatorname{Spl}_K(f_c^{\n})
\quad\text{and}\quad
\mathcal{G}_n=\operatorname{Gal}(K_n/K).
\)
Then for every \(n\ge1\), we have
\[
\mathcal{G}_n\cong
\underbrace{C_d\wr C_d\wr\cdots\wr C_d}_{n\text{ factors}}.
\]
Equivalently, the arboreal Galois representation attached to \(F\), and hence
also to \(f_c\), has full cyclic wreath-product image at every level.
\end{thm}

\begin{proof}
We proceed by induction on \(n\). For \(n=1\), Theorem \ref{thm:F-iterates-irreducible} implies that \(F(x)=x^{d}+C\) is irreducible over \(K\). Since \(\mu_d\subset K\), Theorem \ref{Kummer-Theory} yields \( \mathcal G_1 \cong C_d.\)  In particular, \(\mathcal G_1\) acts transitively on \(R_1\). Assume now that \(n\ge 2\) and that
\[\mathcal G_{n-1}\cong \underbrace{C_d\wr \cdots \wr C_d}_{n-1\text{ factors}}.\] Then \(\mathcal G_{n-1}\) acts transitively on \(R_{n-1}\). By Proposition \ref{lem:primitive_place_F} with \(n-1\), there exists a finite place \(P\) of \(K\) such that
\[
v_P(F^{\n}(0))=1,
\textrm{ and }
v_P(F^{\m}(0))=0
\quad \textrm{ for }1\le m<n.
\]
In particular, $v_P(C)=0$. Since \(v_P(F^{\circ n}(0))=1\), we also have \(\gcd\!\bigl(v_P(F^{\circ n}(0)),d\bigr)=1.\) Therefore all hypotheses of Proposition \ref{prop:general-kummer-independence}
are satisfied with \(A=1\) and \(B=C\). Hence the Kummer classes
\([\beta-C],\ \beta\in R_{n-1},\) are linearly independent in \(K_{n-1}^{\times}/(K_{n-1}^{\times})^{d}.\) Proposition \ref{prop:general-kummer-independence} therefore gives
\[
\operatorname{Gal}(K_n/K_{n-1})
\cong
\prod_{\beta\in R_{n-1}} C_d .
\]
Since \( \mathcal G_{n-1}\) acts on \(R_{n-1}\), the natural action on the preimage tree yields
\[
\mathcal G_n
\cong
\left(
\prod_{\beta\in R_{n-1}} C_d
\right)
\rtimes \mathcal G_{n-1}
=
C_d\wr \mathcal G_{n-1}.
\]
Using the induction hypothesis, we conclude that \(\mathcal G_n \cong\underbrace{C_d\wr C_d\wr\cdots\wr C_d}_{n\text{ factors}}.
\)
\end{proof}

\section{Comparing Arboreal and Drinfeld Galois Representations}

The examples in this section illustrate that the relationship between the arboreal Galois representations associated with \(
f_c(x)=cx^{d}+t\) and the adelic Galois representations attached to the corresponding twisted Carlitz modules is more subtle than one might expect. Although the identity \(\rho_t^c(X)=Xf(X)\)
connects the two theories at the first level, their higher towers encode different arithmetic information. In particular, surjectivity of the adelic Galois representation in the sense of Gekeler does not imply maximality of the associated arboreal Galois representation. Conversely, maximal arboreal Galois image does not force surjectivity of the Adelic representation. The adelic surjectivity of $\rho_t^c(x)$ and arboreal maximality of $f_c(x)$ are independent, except for a one-way local implication at t.

\subsection{An Adelic-Surjective but Arboreally Nonmaximal Example}

In this subsection, we construct a twisted Carlitz module whose adelic Galois image is surjective from  Gekeler's defect formula, while the corresponding polynomial fails the primitive simple divisor property and consequently does not have full arboreal Galois image.
 
Although the example lies outside the hypotheses of our maximality theorems, it
illustrates that adelic surjectivity of the associated twisted Carlitz module does not, by
itself, force maximality of the arboreal Galois representation.

\begin{prop}\label{adelsurj}
Let $K=\mathbb{F}_3(t)$, $c=\frac{t^2+t+2}{t(t+1)^2}$, and consider the twisted Carlitz module $\rho_t^c(x)=tx+cx^3$. Then the adelic Galois representation attached to $\rho_t^c$ is surjective.
\end{prop}

\begin{proof}
For $q=3$, the isomorphism class of a twisted Carlitz module depends only on the class of $c$ modulo squares in $K^\times$. Multiplying $c$ by the square $t^2(t+1)^2$, we obtain \
\(\Delta=t^2(t+1)^2c=t(t^2+t+2).
\) Hence \(\deg(\Delta)=3.\) The leading coefficient of $\Delta$ is $1$, so \(k_0=0.\)
Since $q=3$ and $\deg(\Delta)=3$ are both odd, \(k_0^\ast\equiv k_0+\frac{q-1}{2}\equiv0+1\equiv1\pmod 2.\) Thus \(k_0^\ast=1.\)
 Therefore by Corollary \ref{cor-gek}, we get \(\operatorname{def}(\Delta)=\gcd(\deg(\Delta)-1,q-1,k_0^\ast)=\gcd(2,2,1)=1.\) Hence the adelic Galois representation attached to $\rho$ is surjective.
\end{proof}

\begin{prop}
Let $f(X)=cX^2+t$, $c=\frac{t^2+t+2}{t(t+1)^2}$. Then the critical orbit $a_n=f_c^{\n}(0)$ does not satisfy the primitive simple divisor property.
\end{prop}

\begin{proof}
We have \(a_1=f(0)=t.\) Next, \(a_2=f(t)=ct^2+t.\) Substituting the value of $c$, we obtain
\[a_2=\frac{t(t^2+t+2)}{(t+1)^2}+t =\frac{t(t^2+t+2)+t(t+1)^2}{(t+1)^2}.\]
Since $(t+1)^2=t^2+2t+1$ in $\mathbb F_3[t]$, the numerator becomes
\[t\bigl((t^2+t+2)+(t^2+2t+1)\bigr)=t(2t^2+3t+3)=2t^3=-t^3.\]
Hence \(a_2=-\frac{t^3}{(t+1)^2}.\) The only finite prime with positive valuation at $a_2$ is $t$, which already divides $a_1=t$. Thus $a_2$ has no primitive divisor. In particular, $a_2$ has no primitive simple divisor. Therefore the primitive simple divisor property fails.
\end{proof}

\begin{prop}\label{nonmax}
 The second-level arboreal Galois group is strictly smaller than the full wreath product:
\(\mathcal{G}_2 \not\cong C_2 \wr C_2.\)
\end{prop}

\begin{proof}
Let $\beta$ be a root of the first-level equation $f(x)=0$. Then \(\beta^2=-\frac{t}{c}.\) The product of the two second‑level Kummer classes is represented by $\frac{a_2}{c^3}$. Using the formula \(a_2=-\frac{t^3}{(t+1)^2},\)
we obtain
\[\frac{a_2}{c^3}
   = -\frac{t^3}{(t+1)^2c^3}
   = \frac{1}{(t+1)^2}\left(-\frac{t}{c}\right)^3
   = \frac{\beta^6}{(t+1)^2}=\left(\frac{\beta^3}{t+1}\right)^2,\]
which is a square in $K_1=K(\beta)$. Therefore the two second-level Kummer classes satisfy a nontrivial relation in
$K_1^\times/(K_1^\times)^2.$ Hence the subgroup generated by these classes is not isomorphic to \((\mathbb Z/2\mathbb Z )^2.\) By Theorem~\ref{Kummer-Theory},
\( \Gal(K_2/K_1)\simeq \
\textrm{Hom}(\Gamma_2,\mu_2),
\)
where \(\Gamma_2\) is the subgroup generated by the two
second-level Kummer classes. Since \(\Gamma_2\) is not isomorphic to \((\mathbb Z/2\mathbb Z)^2\), we obtain
\(|\Gal(K_2/K_1)|\le 2,\) and consequently $|\mathcal{G}_2|<8$. Since $|C_2\wr C_2|=8$, it follows that $\mathcal G_2\not\cong C_2\wr C_2$.
\end{proof}

\subsection{An Arboreally Maximal Example with Non-surjective Adelic Image}

\begin{prop}\label{arbmaxnonsurj}
Let  \(K = \mathbb{F}_{7}(t),\ d = 6,\ u = t + 1,\)
and set  \(c = u^{-(d - 1)} = (t + 1)^{-5}.\) Consider the polynomial  \(f_c(x) = cx^{d} + t = (t + 1)^{-5}x^{6} + t\)
and the associated twisted Carlitz module  \(\rho_{t}^c(x) = tx + cx^{7}.\) Then the  arboreal Galois image of \(f_c\) is maximal, whereas the full adelic Galois representation of the twisted Carlitz module \(\rho_t^c\) is not surjective.
\end{prop}

\begin{proof}
First, we prove maximality of the arboreal Galois image. Since  \(c^{-1} = u^{d - 1}\), the polynomial \(f_{c}\) is linearly conjugate over \(K\) to  \(F(x) := u^{-1}f_c(ux).\) Thus 
\(F(x) = x^{6} + \frac{t}{t + 1}.\)  The first‑level equation for \(F\) is  \(x^{6} = -\frac{t}{t + 1}.\)
Since \(v_t\!\left(-\frac{t}{t+1}\right)=1,\) the class \(\left[-\frac{t}{t+1}\right]\in K^\times/(K^\times)^6\)  has order \(6\). Hence the first‑level Kummer extension has Galois group \(C_{6}\), and the first level is transitive. Applying Theorem \ref{thm:conjugate-full-wreath} to \(F(x)\), we obtain that for every \(n\geq 1\),
\[
\operatorname{Gal}(K_{n} / K)\cong \underbrace{C_{6}\wr C_{6}\wr\dots\wr C_{6}}_{n\text{ factors}},
\]
Thus, the arboreal Galois image is maximal.\\
We now show that the associated twisted Carlitz adelic image is not surjective. The $K$-isomorphism class of twisted Carlitz modules depends only on the coefficient \(c\) modulo \((K^{\times})^{d} = (K^{\times})^{6}\). Multiplying \(c\) by the sixth power \((t + 1)^{6}\), we obtain the equivalent polynomial representative  \(
\Delta = (t + 1)^{6}c = t + 1.\) Thus \(\rho_t^c\) is isomorphic over \(K\) to the twisted Carlitz module with coefficient \(\Delta = t + 1\). We compute Gekeler's defect for \(\Delta = t + 1\). Here 
\(\deg \Delta = 1.\) The leading coefficient of \(\Delta\) is \(1\), so if \(g\) is a generator of \(\mathbb{F}_{7}^{\times}\), then \(1 = g^{0}\), and hence \(k_{0} = 0\). Moreover, the finite part of \(\Delta\) consists of the single linear factor \(t + 1\) with exponent \(1\), so \(D = 1\). Since \(q = 7\) and \(D = 1\) are both odd, the constant exponent is 
\[
k_{0}^{*}\equiv k_{0} + \frac{q-1}{2}\equiv 0 + 3\equiv 3\pmod {6}.
\]
By Corollary \ref{cor-gek}, \(\operatorname{def}(\Delta) = \gcd (D -1,\; q - 1,\; k_{0}^*) = \gcd (0,6,3) = 3.\)
In particular, \(\operatorname{def}(\Delta) > 1\). Hence, the full adelic Galois image of the twisted Carlitz module has index \(3\) in the expected adelic group, and so the adelic representation is not surjective. Therefore \(f\) has maximal arboreal Galois image, while the associated twisted Carlitz module has non‑surjective adelic Galois representation.
\end{proof}

\begin{rem}
We further compute the defect on \textit{mod-N} representations. When $N=t^m$ for $m\geq 1$, we have by Theorem \ref{Gekeler} that
\[
\operatorname{def}(\Delta,t^m)
=
\operatorname{def}(\Delta,t)
=
\gcd(\deg\Delta-1,\ q-1,\ k_0^\ast,\ k_1)
=
\gcd(0,6,3,1)
=
1.
\]
Hence the \(t\)-adic representation is surjective. On the other hand, for the \((t+1)\)-power levels, the prime divisor
\(P_1=t+1\) divides \((t+1)^m\). Hence the exponent \(k_1=1\) is omitted from
Gekeler's finite-level defect formula. Thus
\[
\operatorname{def}(\Delta,(t+1)^m)
=
\gcd(\deg\Delta-1,\ q-1,\ k_0^\ast)
=
\gcd(0,6,3)
=
3.
\]
Therefore the nontrivial adelic defect occurs at \((t+1)\)-adic representation of $tx+(t+1)^{-5}x^7.$ 
\end{rem}

The above observation inspires us to study the relation between arboreal maximality of $f_c(x)=cx^{q-1}+t$ and $t$-adic surjectivity of $\rho_t^c(x)=tx+cx^q$.

\begin{thm}
\label{arboreal-max}
Let $q=p^e$ be an odd prime power and $d=q-1$.
Let \(\rho_t^c(x)=tx+cx^q\) be the twisted Carlitz module corresponding to $f_c(x)=cx^{q-1}+t$. Suppose that the first arboreal level
of \(f_c\) is maximal, i.e.
\[
\operatorname{Gal}(K(f_c^{-1}(0))/K)\cong C_d.
\]
Then the \(t\)-adic Galois representation attached to \(\rho_t^c\) is surjective, i.e, \(\operatorname{def}(\Delta,t^m)=1 \ \text{for every}\ m\ge1,\) where \(\Delta\in \mathbb F_q[t] \) is the \(d\)-power-free representative of the class of \(c\) in
\(K^\times/(K^\times)^d\). In other words, arboreal maximality of \(f_c\) implies \(t\)-adic surjectivity of the associated twisted Carlitz module $\rho_t^c$.
\end{thm}

\begin{proof}
The first arboreal level is given by
\(f_c(x)=0
\quad\Longleftrightarrow\quad
x^d=-\frac{t}{c}.\)
Since \(\mu_d\subset K,\)
Theorem \ref{Kummer-Theory} shows that
\[
\operatorname{Gal}(K(f_c^{-1}(0))/K)\cong C_d \textrm{
if and only if the class}
\left[-\frac{t}{c}\right]\in K^\times/(K^\times)^d 
\textrm{ has order \(d\).} \] Replacing \(c\) by its \(d\)-power-free representative \(\Delta\in \mathcal R\), this is equivalent to saying that \(\left[-\frac{t}{\Delta}\right]\) has order \(d\) in \(K^\times/(K^\times)^d\). \\We now show that this forces the \(t\)-adic defect to be equal to $1$. Write
\[\Delta=\gamma^{k_0}t^{k_t}\prod_{j=1}^s Q_j^{k_j},\]
where \(\gamma\) is a generator of \(\mathbb F_q^\times\), the \(Q_j\neq t\) are distinct monic irreducible polynomials in \(\mathcal R\), and \(0\le k_0,k_t,k_j<d.\) By Theorem \ref{Gekeler}, the defect at level \(t^m\) depends only on the prime divisors of \(t^m\). Hence \(\operatorname{def}(\Delta,t^m)=\operatorname{def}(\Delta,t) \ \textrm{ for $m\ge1$}.\)\\
It is therefore enough to prove that \(\operatorname{def}(\Delta,t)=1.\) Suppose, toward a contradiction, that \(\operatorname{def}(\Delta,t)>1.\) Let \(\ell\) be a prime divisor of \(\operatorname{def}(\Delta,t)\), then \(\ell\mid d\). We firstly suppose that \(t\nmid\Delta\). Then all finite prime factors of \(\Delta\) are coprime to \(t\), so Theorem \ref{Gekeler} gives \(\operatorname{def}(\Delta,t)=\gcd(\deg\Delta-1,d,k_0^\ast,k_1,\ldots,k_s).\) Thus \(\ell\mid k_j\) for every \(j\). Hence \(\deg\Delta\) is divisible by \(\ell\). But the assumption $\ell \mid \operatorname{def}(\Delta,t)$ also gives \(\ell\mid \deg\Delta-1,\) which is impossible. Therefore \(t\mid\Delta\), i.e. $k_t$ is nonzero. \\
Now from the properties \(\ell\mid k_j \text{ for every }1\leq j\leq s, \ \ell\mid \deg\Delta-1,\) and
\[\deg\Delta=k_t+\sum_{j=1}^s k_j\deg Q_j,\]
we obtain \(k_t\equiv1\pmod\ell.\)
Since $\ell\mid k_t-1$ and $\ell\mid k_j$ for all $j$, for any finite place $P$ of $K$, the valuation $v_P$ of \(-\frac{t}{\Delta}\) is divisible by \(\ell\).
The condition \(\ell\mid k_0^\ast\) is precisely the remaining constant-factor condition in Gekeler's defect formula, see section 2.5.2 for definition of $k_0^*$. It ensures that the remaining constant factor of $-\frac{t}{\Delta}$ in \(\mathbb F_q^\times\) is also an \(\ell^{th}\) power.
Therefore we have \(-\frac{t}{\Delta}\in (K^\times)^\ell,\) and the class
\[\left[-\frac{t}{\Delta}\right]\in K^\times/(K^\times)^d
\]
has order strictly smaller than \(d\), contradicting the maximality of the first arboreal level. Hence \(\operatorname{def}(\Delta,t)=1,\) the \(t\)-adic Galois representation of the associated twisted Carlitz module $\rho_t^c(x)$ is surjective.
\end{proof}

\appendix

\section{An Eventual Primitive Simple Divisor Theorem}

The primitive simple divisor arguments of Sections \ref{Sec.AMf_m} and \ref{Sec.Irreducible Conjugate} are based on Theorem \ref{thm:mason-stothers} applied to a suitable normalized critical orbit sequence. Inspired by \cite{T13}, the purpose of this appendix is to record a more general version of this argument for arbitrary polynomial coefficients. Although the result is not needed for the proofs of the main theorems, it illustrates that the primitive simple divisor phenomenon is not restricted to the special families considered in the paper.

\begin{thm}
\label{thm:eventual_PSD}
Let $q=p^e$ be an odd prime power and $d\geq 3$ with $p\nmid d$. Let \(f_c(x)=cx^d+t\in K[x],\) where $c\in \mathbb F_q[t]\setminus\{0\}$. Define \(a_n=f_c^{ \n}(0),\) and set \(h:=ct^{d-1}.\) Assume that \(h\notin \mathbb F_q[t^p].\) There exists an integer $N_0\ge 4$, depending on $c$, $q$, and $d$, such that for every $n\ge N_0$, the element $a_n$ possesses a primitive simple divisor. More precisely, for every $n\ge N_0$, there exists an irreducible polynomial $P_n\in\mathbb F_q[t]$ that satisfies \(v_{P_n}(a_n)=1\) and \(P_n\nmid cta_1a_2\cdots a_{n-1}.\)
\end{thm}

\begin{proof}
Write \(a_n=tb_n.\) Then $b_1=1$, and for $n\ge2$,
\(b_n=hb_{n-1}^{d}+1,\ h=ct^{d-1}.\) In particular, \(\gcd(b_n,h)=1,\ \gcd(b_n,b_{n-1})=1.\) \\Let
\[
D_n:=\deg(b_n),\ B_n:=\deg\operatorname{rad}(b_n), \
H:=\deg(h),\ S:=\deg\operatorname{rad}(h).
\]
Since \(D_1=0,\ D_n=H+dD_{n-1},\)
we obtain
\[D_n = H(1+d+\cdots+d^{n-2}) = H\frac{d^{n-1}-1}{d-1}.\]
Applying Theorem \ref{thm:mason-stothers} to \(b_n-hb_{n-1}^{d}=1,\) but since \(h\notin\mathbb F_q[t^p],\) we obtain \[D_n \le \deg\operatorname{rad}(hb_{n-1}b_n)-1 \implies D_n \le S+B_{n-1}+B_n-1.\] Since $B_{n-1}\le D_{n-1}$, it follows that \(B_n \ge D_n-D_{n-1}-S+1.\) Using the degree formula, \(D_n-D_{n-1}=Hd^{n-2},\) we obtain \(B_n \ge Hd^{n-2}-S+1.\) Let $N_n$ denote the degree of the primitive radical of $b_n$. As in the proof of Theorem~\ref{thm:PSD}, every nonprimitive irreducible divisor of $b_n$ divides \[\prod_{\substack{r\mid n \\ r<n}} b_r.\]
Therefore \(N_n \ge B_n - \sum_{\substack{r\mid n \\ r<n}} D_r.\) Consequently, \(N_n \ge Hd^{n-2}-S+1 - \sum_{\substack{r\mid n \\ r<n}} D_r.\) Since the sum of the proper-divisor contributions grows like $d^{n/2}$, whereas the leading term grows like $d^{n-2}$, there exists an integer $N_0\ge4$ such that \(N_n>\frac{D_n}{2}\) for all $n\ge N_0$.\\
Fix $n\ge N_0$. If every primitive irreducible divisor of $b_n$ occurred with multiplicity at least $2$, then the primitive part of $b_n$ would have degree at least \(2N_n>D_n,\)
which is impossible. Hence, there exists a primitive irreducible divisor $P_n$ of $b_n$ such that \(v_{P_n}(b_n)=1.\) Since $a_n=tb_n$ and $P_n\neq t$, we obtain \(v_{P_n}(a_n)=1.\) Moreover, \(P_n\nmid cta_1a_2\cdots a_{n-1}.\) Therefore $P_n$ is a primitive simple divisor of $a_n$.
\end{proof}

\begin{cor}
\label{cor:eventual_Kummer}
Keep the notation and hypotheses of Theorem ~\ref{thm:eventual_PSD}. Let $d\mid q-1$. Assume that $n\geq N_0$, that $f_c^{\m}$ is separable over $K$ for $1\leq m\leq n$, and that \(\mathcal G_{n-1}\) acts transitively on $R_{n-1}$. Then
\[
\operatorname{Gal}(K_n/K_{n-1})
\cong
\prod_{\beta\in R_{n-1}} C_d.
\]
\end{cor}
\begin{proof}
For each $\beta\in R_{n-1}$, the equation \(f_c(x)=\beta\) is equivalent to \(x^d=\frac{\beta-t}{c}.\) Since \(\mu_d\subset K\subset K_{n-1}\), the extension \( K_n/K_{n-1} \) is a multiradical Kummer extension obtained by adjoining the \(d^{th}\) roots of the elements 
\[ \frac{\beta-t}{c},\quad \textrm{where}\quad \beta\in R_{n-1}. \]
By Theorem ~\ref{thm:eventual_PSD}, $a_n$ has a primitive simple divisor $P_n$ satisfying
\(P_n\nmid cta_1a_2\cdots a_{n-1}.\) The same valuation-separation argument as in Proposition \ref{prop:general-kummer-independence} shows that the classes \(\left[\frac{\beta-t}{c}\right],\textrm{ with }\ \beta\in R_{n-1},\) are linearly independent in \(K_{n-1}^{\times}/(K_{n-1}^{\times})^d.\) Therefore Theorem \ref{Kummer-Theory} gives
\[\operatorname{Gal}(K_n/K_{n-1})\cong\prod_{\beta\in R_{n-1}} C_d.\]
\end{proof}

% ------------------------------------------------------------
% Bibliography
% ------------------------------------------------------------
% Option 1: use a separate .bib file. Rename the bibliography file as needed.
\bibliographystyle{alpha}
\bibliography{MAGIPTCT}

% Option 2: use an inline bibliography instead of a .bib file.
% Uncomment the following block if you do not want to use BibTeX.
%
% \begin{thebibliography}{99}
%
% \bibitem[Lan84]{Lan84}
% S. Lang, \emph{Algebra}, Addison-Wesley, 1984.
%
% \end{thebibliography}

\end{document}